\documentclass{article}	

\usepackage{amsmath,amsthm,amsfonts,amscd} 
\usepackage{eucal} 	 	
\usepackage{verbatim}      	
\usepackage{makeidx}       	
\usepackage{psfig}         	
\usepackage{graphicx}
\usepackage{dsfont}
\usepackage{epsfig}         	
\usepackage{url}		
\usepackage[small,nohug]{diagrams}
\diagramstyle[labelstyle=\scriptstyle] 

\newtheorem{thm}{Theorem}[section]
\newtheorem{cor}[thm]{Corollary}
\newtheorem{lem}[thm]{Lemma}

\theoremstyle{definition}
\newtheorem{defn}{Definition}[section]

\theoremstyle{remark}

\begin{document}

\title{Pattern Equivariant Representation Variety of Tiling Spaces for Any Group G}

\author{Haydar O\={g}uz Erdin}
\date{May 2, 2009}
\maketitle

\index{Abstract}
\begin{abstract}
It is well known that the moduli space of flat connections on a trivial principal bundle $M\times G$, where $G$ is a connected Lie group, is isomorphic to the representation variety $Hom(\pi_1(M), G)/G$.For a tiling $T$, viewed as a marked copy of $\mathbb{R}^d$, we define a new kind of bundle called \textit{pattern equivariant bundle} over $T$ and consider the set of all such bundles. This is a topological invariant of the tiling space induced by $T$, which we call \textit{PREP(T)}, and we show that it is isomorphic to the direct limit $\varinjlim_{f_n} Hom(\pi_1(\Gamma_n), G)/G$, where $\Gamma_n$ are the approximants to the tiling space. $G$ can be any group. As an example, we choose $G$ to be the symmetric group $S_3$ and we calculate this direct limit for the Period Doubling tiling and its double cover, the Thue-Morse tiling, obtaining different results. This is the simplest topological invariant that can distinguish these two examples.
\end{abstract}

\section{Introduction}
In this paper we generalize an idea first mentioned by Sadun in \cite{Sadun1}. It is well known that given a connected group $G$ and a manifold $M$, the moduli space of flat connections on the trivial principal bundle $M\times G$ is isomorphic to the \textit{representation variety} $Hom(\pi_1(M), G)/G$, where the quotient is by conjugation. When $G$ is abelian, this reduces to $H^1(M, G)$, so we restrict attention to non-abelian $G$. Working on tiling spaces, we consider flat, pattern equivariant (PE) connections on the trivial principal bundle $\mathbb{R}^d\times G$, modulo PE gauge transformations and call this the \textit{PE representation variety} of the tiling $T$. Then, by considering these PE connections as pullbacks of zero curvature connections on the approximants $\Gamma_n$, it is not hard to see that the PE representation variety of the tiling $T$ is the direct limit $\varinjlim_{f_n} Hom(\pi_1(\Gamma_n), G)/G$, where $f_n$ is the dual of the forgetful map between the approximants $\Gamma_n$.

The original definition was for connected Lie groups, but calculation of this direct limit is easiest when the group $G$ is finite and hence disconnected. In section \ref{examples}, we take $G=S_3$ and calculate it for the Period Doubling tiling and its double cover, the Thue-Morse tiling. This calculation shows that, even though these two tiling spaces have the same cohomology, the PE representation variety distinguishes them. The computer algebra program Magma can be used to calculate $Hom(H, G)/G$ for any finitely generated free group $H$ and symmetric group $G$, which will be a lot of help in complicated examples. The relevant instructions on how to implement this on Magma is given after the examples. In order to make the generalization to finite groups, we change the definition of the PE representation variety by defining it not in terms of connections but in terms of bundles. We call it $PREP(T)$ and define it in \ref{prep}. We use $\mathbb{R}^2$ in the definitions but the generalization to $\mathbb{R}^d$ is immediate.   Using this definition, the main result of the paper is proved in section \ref{mainresult}:
\begin{thm}\label{datheorem}
Let $T$ be a tiling space and $K_t$, $t\in\mathbb{R}$ be its G\"{a}hler approximants. Let $t_1\leq t_2\leq...$ be an increasing sequence and $f:K_{t_i}\rightarrow K_{t_j}$ be the forgetful map where $t_i$ is bigger than $t_j$. Then,
  \begin{equation}
  PREP(T) = \varinjlim_{f^*} Hom(\pi_1(K_{t_n}), G)/G.
  \end{equation}
\end{thm}

Once the right definitions have been chosen, many of the proofs are immediate (and frequently omitted). The sections before \ref{mainresult} work toward the proof of this result. In section \ref{invariance} we show the invariance of the PE representation variety under homeomorphisms of tiling spaces using a theorem of Rand \cite{Rand}.

A finer invariant can be obtained without modding out by $G$ in the direct limit. To achieve this, we consider only based bundles in the definition of $PREP(T)$ and call it the $PREP_b(T)$. This is the definition in \ref{basedbundle} and in section \ref{gamma} we have

\begin{thm}\label{datheorem2}
Let $T$ be a tiling space and $K_t$, $t\in\mathbb{R}$ be its G\"{a}hler approximants. Let $t_1\leq t_2\leq...$ be an increasing sequence and $f:K_{t_i}\rightarrow K_{t_j}$ be the forgetful map where $t_i$ is bigger than $t_j$. Then,
  \begin{equation}
  PREP_b(T) = \varinjlim_{f^*} Hom(\pi_1(K_{t_n}), G).
  \end{equation}
\end{thm}


\section{Two Examples}\label{examples}

For both of the examples below, the approximants of the tilings are the wedge of two circles. Hence their fundamental group is the free group on two generators, $F_2$. In both examples, our group $G$ will be the simplest non-abelian group, that is the dihedral group of order 6: $S_3=\{1, a, a^2, b, ab, a^2b\}$. So we need to calculate $Hom(F_2, S_3)/S_3$ and the only difference in the two examples will be the forgetful map in the direct limit. A two page hand calculation shows that

\begin{equation}
Hom(F_2, S_3)/S_3=\{\phi_{1,1}, \phi_{1,a},\phi_{1,b},\phi_{a,1},\phi_{b,1},\phi_{a,a},\phi_{a,a^2},\phi_{a,b},\phi_{b,a},\phi_{b,b},\phi_{ba,ab}\},
\end{equation}
where $\phi_{x,y}$ denotes the conjugacy class of a homomorphism which takes one generator of $F_2$ to $x$ and the other to $y$ in $S_3$.

\subsection{Example 1: Thue-Morse}
The substitution rule for the Thue-Morse tiling is given by $\sigma(\alpha)=\alpha\beta$, $\sigma(\beta)=\beta\alpha$. The dual of the forgetful map $f$ is the substitution map itself thought of as a homomorphism from $F_2$ to $S_3$. We have to trace each of the 11 elements of $Hom(F_2, S_3)/S_3$ and see which ones are identified. A typical calculation goes as follows: $$f^*\circ \phi_{1, b}(\alpha)=\phi_{1, b}(\sigma(\alpha))=\phi_{1, b}(\alpha\beta)=1\cdot b=b.$$ A similar calculation with the other generator gives $\phi_{1, b}(\beta\alpha)=b\cdot 1=b$. Hence $\phi_{1, b}$ is mapped to $\phi_{b, b}$ in the direct limit. Repeating this for all the other elements shows that

$$\varinjlim_{f^*} Hom(F_2, S_3)/S_3=\{\phi_{1, 1}, \phi_{a, a}\}.$$

\subsection{Example 2: Period Doubling}
The substitution rule for the Period Doubling tiling is given by $\sigma(\alpha)=\beta\beta$, $\sigma(\beta)=\alpha\beta$. A similar calculation to the above shows that

$$\varinjlim_{\sigma^*} Hom(F_2, S_3)/S_3=\{\phi_{1, 1}, \phi_{1, b}, \phi_{a, a}, \phi_{1, a}, \phi_{a^2, a},\phi_{a^2, 1}\}.$$

Comparing the two direct limits, we see that it differentiates between the Period Doubling and its double cover, the Thue-Morse tiling space. This is the simplest topological invariant that tells them apart.

\subsection{Magma Code For Calculating Representation Variety}

For any finitely generated free group F and symmetric group S, we will give instructions for how to compute $Hom(F, S)/S$ using the computer algebra program Magma. The example we give is for a free group generated by two elements and symmetric group of order six but the generalization is trivial. After each step below hit the \emph{Enter} key (the semi-colon is typed into command line too).
\begin{enumerate}
\item Define the free group: F:=FreeGroup(2);
\item Define the symmetric group: S:=SymmetricGroup(3);
\item Calculate $Hom(F, S)/S$: RV:= Homomorphisms(F, S: Surjective := false);
\item Show the cardinality of $RV$: \#homs;
\end{enumerate}

After these steps, the command line should return ``11``, which was the result we obtained in the examples above.


\section{PE Bundles}\label{pedef}

The idea of modifying cohomology to obtain PE cohomology so that it is applicable to tilings was first given by Kellendonk and Putnam in \cite{Kel} and \cite{Put}. We use the same idea in the category of bundles. One can intuitively picture a $G$-bundle as a geometric object (say an interval) glued to itself  along another geometric object (say a circle) but with a twist dictated by the group $G$. The intuitive idea behind PE bundles over a tiling $T$ is that this twist should also depend on the patterns in the tiling $T$. That is, if two points have the same pattern of tiles up to a certain range around them, them we should invoke the same group element which will dictate the twisting. For this we need the following definitions:

\begin{defn}[PE open set]\index{pattern equivariant open set}
An open set $U$ of $\mathbb{R}^2$ is said to be a \textit{PE open set with radius $t>0$} if for every $x\in U$ it contains all $y\in \mathbb{R}^2$ such that $[B_t(x)]-x = [B_t(y)]-y$, where $[B_t(x)]$ is the patch of the tiling $T$  containing the ball $B_t(x)$ with center $x$ and radius $t$.
\end{defn}

\begin{defn}[PE open cover]\index{pattern equivariant open cover}
An open cover $\{U_i\}$ of $\mathbb{R}^2$ is said to be a \textit{PE open cover with radius $t>0$} if  every $U_j\in \{U_i\}$ is a PE open set with radius $t$.
\end{defn}

\begin{defn}[PE fiber bundle]\index{pattern equivariant fiber bundle}
Let $F$, $P$ be topological spaces and $\pi: P \rightarrow  \mathbb{R}^2$ be an onto continuous map. $P$ is called a \textit{PE fiber bundle over T with radius} $t$ and fiber $F$ if there exists a PE open cover $\{U_i\}$ of $\mathbb{R}^2$ with radius $t$ such that for every $x\in \mathbb{R}^2$ there exists a $U_j\in \{U_i\}$ with the following properties:

  \begin{enumerate}
    \item $U_j \times F$ is homeomorphic to $\pi^{-1}(U_j)$ by a homeomorphism $\phi_j$.
    \item The following diagram is commutative:
  \end{enumerate}
\end{defn}

\begin{diagram}
 U_j \times F & \rTo^{\phi_j}  & \pi^{-1}(U_j) \\
                       & \rdTo_{p_1}  & \dTo_{\pi}\\
                       &                        & U_j
\end{diagram}
The pairs $(U_j, \phi_j)$ are called \textit{charts} as usual.

\begin{defn}[PE function]
Let $T$ be a tiling of $\mathbb{R}^2$. A function $f:\mathbb{R}^2 \rightarrow \mathbb{R}$ is called a \textit{PE function with radius $t$} if for any $x, y\in T$ such that $[B_t(x)]-x=[B_t(y)]-y$, we have $f(x)=f(y)$.
\end{defn}

\begin{defn}[PE G-bundle]\index{pattern equivariant G-bundle}
Let $P$ be a PE fiber bundle with radius $t$ as in the above definition. Let $F$ be its fiber space and assume there is a left $G$ action on $F$. If for every pair of charts $(U_i, \phi_i)$ and $(U_j, \phi_j)$ the map  $$\Phi := \phi_j^{-1} \circ \phi_i : (U_j \cap U_i)\times F \rightarrow (U_j \cap U_i)\times F$$ has the form $\Phi(u,f)=(u,\; h_{ji}(u)\cdot f)$, where $h_{ji}: (U_j \cap U_i) \rightarrow G$ is a PE function with radius $t$ and the $\cdot$ is the action of $G$ on the fiber $F$, then we call $P$ a \textit{PE G-bundle over T with radius t}. The maps $h_{ji}$ are called the \textit{transition functions}\index{transition functions}.
\end{defn}

\begin{defn}[PE principal bundle]\index{pattern equivariant principal bundle}
If the PE G-bundle $P$ with radius $t$ has $G$ itself as its fiber space with the left multiplication in the group as the group action and if the induced action on the fibers $$\mu:\pi^{-1}(x)\times G\rightarrow \pi^{-1}(x)$$ is free and transitive then $P$ is called a \textit{PE principal bundle over T with radius t and group G}.
\end{defn}

Isomorphism of two PE bundles over the tiling $T$ is defined as in the usual bundle theory with the only difference being the requirement that the transition functions must be PE.

As in the usual bundle theory, we also have the fact that a refinement of  a trivializing neighborhoods is also a trivializing neighborhood. Hence, given two PE G-bundles, we can assume that they have the same PE open cover by looking at the intersection of their covers. This gives us the following characterization whose proof can be obtained with minor modifications to the non-PE case (see Naber's book \cite{Naber}).

\begin{lem}[Bundle isomorphism lemma\index{Bundle isomorphism lemma}]
Let $P_1$ and $P_2$ be two PE G-bundles with the same PE open cover $\{U_i\}$ of radius $t$ and the same fiber $F$ over $\mathbb{R}^2$. Let $h_{ji}$ and $\grave{h}_{ji}$ be their transition functions respectively. Then, $P_1$ and $P_2$ are isomorphic if and only if there exists PE functions $\lambda_j: U_j\rightarrow G$ with radius $t$ such that
\begin{equation}\label{eq:eqvl}
\grave{h}_{ji}(x)=\lambda_j(x)^{-1}h_{ji}(x)\lambda_i(x),
\end{equation}
for all $x\in U_i\cap U_j$.
\end{lem}

\begin{thm}[Bundle construction theorem\index{Bundle construction theorem}]
Let $\{U_i\}, i\in J$ be a PE open cover of $\mathbb{R}^2$ with radius $t$ where $J \subseteq \mathbb{N}$ is an index set, let $F$ be a topological space and let $G$ be a topological group which acts on $F$. Assume for every $U_i \cap U_j \neq \emptyset$ there exists continuous PE functions $h_{ji}:U_j \cap U_i \rightarrow G$ with radius $t$ such that for every  $x\in U_i \cap U_j \cap U_k$ we have 
\begin{equation}
h_{kj}(x)\cdot h_{ji}(x) = h_{ik}(x)
\end{equation}
Consider the space $\mathbb{R}^2 \times F \times J$ and its subset $T:=\{(x,g,i)\in \mathbb{R}^2\times F \times J : x\in U_i\}=\coprod_{i} U_i\times F \times \{i\}$. Define an equivalence relation $\leftrightarrow$ on T as follows: $$(x, f, i)\leftrightarrow(\acute{x}, \acute{f}, j) \Leftrightarrow \acute{x}=x\; and\; \acute{f}=h_{ji}(x)\cdot f$$ Let $P$ be the set of equivalence classes of this equivalence relation. Then, $P$ is a PE G-bundle with radius $t$ and is unique up to isomorphism of PE bundles.
\end{thm}


\begin{defn}[flat G-bundle]\index{flat G-bundle}
A G-bundle is called \textit{a flat G-bundle} if all the transition functions are locally constant functions.
\end{defn}

Notice that if $P$ is a PE bundle with radius $t$ over $T$ then it is also a PE bundle for any radius $t'$ which is bigger than $t$. This leads to the following equivalence relation:
\begin{defn} Let $P$ and $R$ be PE bundles  over the tiling $T$ with radii $p$ and $r$. If there exists an $s$ bigger than both $p$ and $r$, and if there exists an isomorphism between $P$ and $R$ as $s$ bundles, then we consider them to be equivalent PE bundles.
\end{defn}

We can now define the crucial object and state the main result that will be our focus.

\begin{defn}[PE representation variety]\label{prep}
Let $S$ be the set of all flat PE principal $G$-bundles $P$ over the tiling $T$ of $\mathbb{R}^2$. The set of equivalence classes in $S$ with respect to the above equivalence relation is called the \textit{PE representation variety of} $T$ and denoted $PREP(T)$. 
\end{defn}

Now we will  define \emph{based bundles} which will give us a finer invariant then $PREP(T)$.
\begin{defn}[based G-bundle]\label{basedbundle}
Let $(P, p_0)$ be a G-bundle with a fixed point $p_0$ chosen. Assume also that all the transitions functions arising from the chart $U_{p_0}$ around $p_0$ are constant; that is $g_{pp_0}=1$. Then, we call $(P, p_0)$ a \textit{based G-bundle}\index{based G-bundle} with the \textit{base point} $p_0$.
\end{defn}

\begin{defn}[based pattern equivariant representation variety]\label{basedprep}
Let $S$ be the set of all flat based PE principal $G$-bundles $P$ over the tiling $T$ of $\mathbb{R}^2$. The set of equivalence classes in $S$ with respect to the same 
equivalence relation as in the definition \ref{prep}  is called the \textit{ Based-PE representation variety of} $T$ and denoted $PREP_b(T)$. 
\end{defn}


\section{PE bundles and approximants}\label{bundlesapproximants}
\index{Arrangement of Dissertation@\emph{Arrangement of Dissertation}}%

The definition we will use for the approximants is a continuous generalization of G\"{a}hler's construction with collared tiles first introduced in \cite{Gahler}. Consider the tautological $\mathbb{R}^2$ bundle  $E\rightarrow \Omega$ over the tiling space $\Omega$ where the fiber over a $T\in\Omega$ is  $T$ itself as a marked copy of $\mathbb{R}^2$. We define an equivalence relation on the total space $E$ as follows:

\begin{defn}
Let $t\in \mathbb{R}$ and $x$, $y$ in $E$ be such that $x\in T$, $y\in T'$. Then, $x\leftrightarrow_t y$ if $[B_t(x)]_T-x=[B_t(y)]_{T'}-y.$
\end{defn}

\begin{defn}[Approximants for $\Omega$]
The quotient of $E$ with respect to the above equivalence relation is called the \textit{t approximant of \;$\Omega$} and is denoted by $K_t$.
\end{defn}

By sending a $T\in\Omega$ to the equivalence class of $0\in T$ we obtain a natural projection $\pi_t:\Omega \rightarrow K_t$. Thus, $K_t$ is the set of all possible instructions for layering tiles around the origin up to distance $t$, generalizing Gahler's construction to a continuous radius $t$. For $t_2\geq t_1$ we have the \textit{forgetful map}\index{forgetful map} $f: K_{t_2}\rightarrow K_{t_1}$ which forgets the pattern between the radii $t_1$ and $t_2$ and sends the pattern up to $t_2$ to the pattern up to $t_1$. If we choose a sequence of increasing radii $t_1\leq t_2\leq\ldots$ going to infinity, then a point in the inverse limit space $S=\stackrel{lim}{\leftarrow}(K_{t_i}, f)$ will be a consistent instruction for tiling all of $\mathbb{R}^2$. This gives a bijection between $\Omega$ and $S$, which can be shown to be a homeomorphism.

We will use the same letter $\pi$ to denote the restriction of the natural projection $\pi_t:\Omega\rightarrow K_t$ to the translational orbit of $T$. Hence the restriction is denoted by $\pi_t:\mathbb{R}^2\rightarrow K_t$. We will see that the set of all PE G-bundles over a tiling $T$ of $\mathbb{R}^2$ is equal to the union of the sets of pullbacks of all the G-bundles on all of the approximants $K_t$.

\begin{lem}
Let $K_t$ be the radius $t$ approximant to the tiling $T$ and let $\pi_t : \mathbb{R}^2 \rightarrow K_t$ be the restriction of the natural projection. Let $\{V_i\}$ be an open cover of $K_t$. Then, $\pi_t^*\{V_i\}$ is a PE open cover of $\mathbb{R}^2$ with radius $t$.
\end{lem}

Since pullbacks of functions on $K_t$ are PE, this lemma shows that pullbacks of bundles on $K_t$ are PE bundles on $\mathbb{R}^2$ with radius $t$. 

\begin{lem}
Let $U$ be a PE open set of $\mathbb{R}^2$ with radius $t$. Then, there exists an open set $V$ of $K_t$ such that $U=\pi_t^{-1}(V)$.
\end{lem}

By this lemma, given a PE bundle $P$ with radius $t$ on $\mathbb{R}^2$, we can find a corresponding open cover on $K_t$. Since transition function of $P$ are pullbacks of functions on $K_t$, we can use this open cover and these functions to construct a bundle on $K_t$ by the standard bundle construction theorem of the non-PE bundle theory. Then, the pullback of this bundle to $\mathbb{R}^2$ is $P$. So, we have proved the following theorem:

\begin{thm}\label{thm:aaa}
Let $S$ be the set of all PE G-bundles of $T$ and let $\{K_t\}$ be the set of all approximants so that the tiling space of $T$ is the inverse limit of $K_t$s. Then, $S$ is the union of the sets of pullbacks of all the G-bundles on all of the $K_t$s.
\end{thm}


\section{Representation Variety of approximants}\label{gamma}

In this section we will see that the set of all flat principal bundles with group $G$ on a given $K_t$ is isomorphic to $Hom(\pi_1(K_t), G)/G$. The latter is called the \textit{representation variety}\index{representation variety of approximants} of $K_t$ and  modding by the $G$ means the conjugacy classes of the homomorphisms into $G$. Since the arguments we will give can be applied to any flat G-bundle, we will state and prove the claims for a general base space $X$, not just for the approximants $K_t$. The results here are proven by Steenrod in \cite{Steenrod} for the case of totally disconnected groups. We will prove them for any flat $G$-bundle. The way we prove them also reveals that it is easy to modify them to get the similar results for the based bundles.

The main theorem that we want to prove is the following:

\begin{thm}\label{thm:exercise}
Let $P$ be a flat G-bundle over $X$ with fiber $F$. Then, it induces a homomorphism from $\pi_1(X)$ to $G$. Conversely, any such representation is induced by a flat G-bundle over $X$ with fiber $F$.
\end{thm}

The following lemma and its proof is standard material but fundamental for us. For completeness sake we present the proof in the appendix.

\begin{lem}\label{lem:ma232}
Let $P$ be a flat G-bundle over $X$ with fiber $F$ and let $X'$ be the universal cover of $X$ with the covering map $f:X'\rightarrow X$. Then, the pullback bundle $f^*(P)$ is trivial.
\end{lem}
\begin{diagram}
f^*(P)      &\rTo^{f'}& P         \\
\dTo^{\pi'} &         & \dTo^{\pi}\\
X'          &\rTo^{f} & X
\end{diagram}



\begin{figure}
\includegraphics[width=95mm]{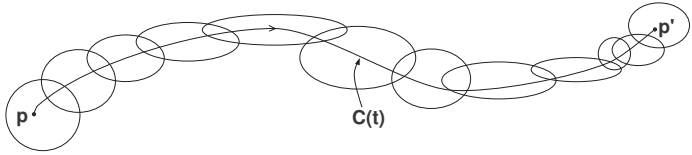} \caption{Multiplying transition functions}
\end{figure}

Given a curve in $X'$, if we just consider the multiplication of the transition functions, see figure 1, then we obtain a map $\psi$ from curves in $X'$ to $G$. Now consider the fundamental group $\pi_1(X, x_0)$ of $X$ based at $x_0\in X$. If we consider lifts of elements of $\pi_1(X, x_0)$ to a $p\in X'$ where $f(p)=x_0$ then $\psi$ will give us a homomorphism from $\pi_1(X, x_0)$ to $G$. It is easy to see that if we change the base point of $\pi_1(X, x_0)$ to a $x_1$ then the homomorphism will change by an inner automorphism of $G$. Hence, we proved the first half of Theorem \ref{thm:exercise}.

For the converse, let $\rho: \pi_1(X)\rightarrow G$ be a homomorphism, let $X'$ be the universal cover of $X$ with the covering map $f$ and let $G$ act on $F$. Define an equivalence relation on $X'\times F$ by $(x', a)\leftrightarrow(x'\cdot \alpha, \rho(\alpha)^{-1}\cdot a)$ for every $\alpha\in\pi_1(X)$ where the action on the first slot is the usual monodromy action of the $\pi_1(X)$ on the fibers of the universal cover $X'$. Modding out by this equivalence relation gives us the associated bundle $X'\times_\rho F$ with the group $G$ and fiber $F$. Since the transition functions of this bundle is given by $\rho(\alpha)$ and since $\pi_1(X)$ is discrete, $X'\times_\rho F$ is flat and this proves the other half of Theorem \ref{thm:exercise}.

This theorem shows that we have an onto map $\gamma$ from the set $S$ of all flat G-bundles over $X$ to $Hom(\pi_1(X), G)$: $$\gamma: S\rightarrow Hom(\pi_1(X), G).$$ We remarked above that changing the base point of the fundamental group changes this map  by an inner automorphism of $G$. Keeping the base point fixed but changing the initial choice of the trivializing chart around this point changes the homomorphisms in the same way. To see this recall that given a curve $C$ in $X'$, $\Psi$(the map in the proof of Lemma 2.3.2) does not depend on the middle open sets that cover $C$ and hence $\psi$ doesn't depend on them too. But if $C$ is a lift of an $\alpha\in\pi_1(X, x_0)$ then the end points of $C$ will live in the lift of the same open set $U_0$ around $x_0$. If we choose a different chart $U_1$ around $x_0$, then by the co-cycle condition we get $\tilde{\psi}=g_{10}\cdot\psi \cdot g^{-1}_{10}$ where $g_{10}$ is the transition function induced by $U_0$ and $U_1$. Hence, $\psi$ is determined up to conjugacy in $G$. Modding $Hom(\pi_1(X), G)$ by $G$ takes care of both of these choice issues. So, let $\chi(P)$ denote the conjugacy class in $G$ determined by the bundle $P$. Then, we have shown that each flat G-bundle $P$ over $X$ determines a unique conjugacy class $\chi(P)$ in $G$ and we have an onto map $\Gamma$ from the set of all flat G-bundles over $X$ to $Hom(\pi_1(X), G)/G$:$$\Gamma:S \rightarrow Hom(\pi_1(X), G)/G.$$

The forward direction of the following lemma will give us the constancy of $\Gamma$ on associated bundles and the converse direction of the lemma will show us that $\Gamma$ is one-to-one if we consider it on the right domain(the set of flat \textit{principle} G-bundles not all flat G-bundles). In the proof of this lemma we will use the following construction:

The fact that the pullback bundle is a product bundle lets us define a \textit{translation of fibers} over $X$. Let $C:I\rightarrow X$ be a curve connecting $x$ to $\tilde{x}$ in $X$ and let $C'$ be its lift to the universal cover $X'$ of $X$. Let $g\in G$ be associated to the path $C'$ given by $\psi$ as above and let $\Omega: f^*(P)\rightarrow X'\times F$ be the equivalence we obtained in Lemma \ref{lem:ma232}. 
\begin{diagram}
X'\times F \cong f^*(P)      &\rTo^{f'}& P         \\
\dTo^{\pi'} 								 &         & \dTo^{\pi}\\
X'          								 &\rTo^{f} & X
\end{diagram}
Since $f^*(P)$ is a pull back bundle, for any points $x'$, $x$ such that $f(x')=x$, the fibers over them are the same. So we have the following diagram:

\begin{diagram}
 \pi'^{-1}(x')&\rTo^{\Omega} &\{x'\}\times F &\rTo^{g}&\{\tilde{x}'\}\times F &\rTo^{\Omega^{-1}}&\tilde{\pi}^{-1}(\tilde{x}')\\
 \uTo&& && &&\dTo\\
\pi^{-1}(x) && \rTo^{h_C} && &&\pi^{-1}(\tilde{x}),
\end{diagram}
where the $g$ on the top row is the map obtained by the right action of $G$ on the fibers. Taking the composition of these maps we have the map $$h_C:\pi^{-1}(x)\rightarrow \pi^{-1}(\tilde{x}).$$

\begin{lem}
Two flat G-bundles $P$ and $P'$ over $X$ with fiber $F$ are associated if and only if $\chi(P)=\chi(P')$.
\end{lem}

\begin{proof}
Let $P$, $P'$ be two flat principal G-bundles over $X$ such that $\chi(P)=\chi(P')$. 
\begin{diagram}
P & & P'\\
 & \rdTo_{\pi} &  \dTo_{\pi'}\\
 & &X
\end{diagram}
Let $x$ be any point in $X$ and let $C:I\rightarrow X$ be a curve connecting $x_0$ to $x$. Choose charts around $x_0$ and $x$ to get the translation of the fibers
\begin{eqnarray}
h_C    &:& \pi^{-1}(x_0) \rightarrow \pi^{-1}(x)\\
h'_C   &:& \pi'^{-1}(x_0)\rightarrow \pi'^{-1}(x),
\end{eqnarray}
and the homomorphisms $\psi$, $\psi':\pi_1(X, x_0)\rightarrow G$. Define $$h_x:\pi^{-1}(x)\rightarrow \pi'^{-1}(x)$$ by $$h_x:= h'_C\phi'_{x_0}\phi^{-1}_{x_0}h^{-1}_C,$$ where, as usual, $\phi$\;s are the homeomorphisms of the fibers with the group $G$ given by the trivializing charts around $x_0$ for $P$ and $P'$. Since $\chi(P)=\chi(P')$, we can choose $\psi$ and $\psi'$ such that $\psi=\psi'$. We will use this fact to show that $h_x$ is independent of the path $C$ we choose.

Let $D$ be another curve connecting $x_0$ to $x$ and let $h_{x,D}$ be the map obtained from $D$. Then,

\begin{eqnarray}
h^{-1}_x h_{x, D} &=& \left(h_C\phi_{x_0}\phi'^{-1}_{x_0}h'^{-1}_C\right)\left(h'_D\phi'_{x_0}\phi^{-1}_{x_0}h_{D}^{-1}\right)\\
                  &=& h_C\phi_{x_0}\left(\phi'^{-1}_{x_0}(h'_{C^{-1}D})\phi'_{x_0}\right)\phi^{-1}_{x_0}h_{D}^{-1}\\
                  &=& h_C\phi_{x_0}\left(\phi'^{-1}_{x_0}(h'^{-1}_{D^{-1}C})\phi'_{x_0}\right)\phi^{-1}_{x_0}h_{D}^{-1}\\
                  &=& h_C\phi_{x_0}\psi'(D^{-1}C)\phi^{-1}_{x_0}h_{D}^{-1}\\
                  &=& h_C\phi_{x_0}\psi(D^{-1}C)\phi^{-1}_{x_0}h_{D}^{-1}\\
                   &=& h_C\phi_{x_0}\left(\phi^{-1}_{x_0}(h^{-1}_{D^{-1}C})\phi_{x_0}\right)\phi^{-1}_{x_0}h_{D}^{-1}\\
                  &=& h_C\phi_{x_0}\left(\phi^{-1}_{x_0}h_{C}^{-1}h_{D}\phi_{x_0}\right)\phi^{-1}_{x_0}h_{D}^{-1}\\
                  &=& id_{\pi^{-1}(x)}.
\end{eqnarray}
Hence, $h_x$ is independent of the curve connecting $x_0$ to $x$. Define $h:P\rightarrow P'$ by $h(b)=h_x(b)$ where $\pi(b)=x$. It is not hard to show that $h$ so defined satisfies the two conditions of being an equivalence of bundles.

For the other direction, given a flat G-bundle $P$ over $X$ and its associated bundle $P'$, we have to show that $\chi(P)=\chi(P')$. But this is trivial since we will be using the same universal cover of $X$ and hence the transition functions are the same in both cases.
\end{proof}

Hence, we finally have the

\begin{cor}
The set of all flat principal bundles with  group $G$ on $K_t$ is isomorphic to $Hom(\pi_1(K_t), G)/G$.
\end{cor}

Recall that we were modding out $Hom(\pi_1(X), G)$ by $G$ because of the double dependence of the map $\gamma$ on the base point $p_0$ of $\pi_1(X)$ and on the triviliazing neighborhood $U_{p_0}$ around $p_0$. If we fix this base point and require all the transition functions to be $g_{pp_0}=1$, then using the definitions \ref{basedbundle} and \ref{basedprep}, it is not hard to modify the proofs given in this section to get the following:
\begin{thm}
The set of all flat principal based G-bundles over $X$ are isomorphic to $Hom(\pi_1(X), G)$.
\end{thm}

\begin{cor}
The set of all flat principal based bundles with  group $G$ on $K_t$ is isomorphic to $Hom(\pi_1(K_t), G)$.
\end{cor}

\section{The Main Result}\label{mainresult}

Using the results we developed so far, we can now obtain the main result we want. Let $B_t$ be the set of all flat PE principal bundles with radius $t$ over $T$ and let $S_t$ be the set of all flat bundles over the approximant $K_t$ for the tiling $T$. Choose a sequence of radii $t_1\leq t_2\leq \ldots$ going to infinity. Then, we have the following commutative diagram:

\begin{diagram}
B_{t_1}     & \rInto^{i} & B_{t_2}      & \rInto^{i}    & B_{t_3}     & \rInto^{i}   & \cdots\\
\uTo_{\pi^*}  &            & \uTo_{\pi^*}   &             & \uTo_{\pi^*}  &            & \cdots\\
S_{t_1}     & \rTo_{f^*} & S_{t_2}		  & \rTo_{f^*}  & S_{t_3}	  	& \rTo_{f^*} & \cdots  ,
\end{diagram}
where $i$ is the inclusion map, $\pi^*$ is the pullback map obtained from the projection map and the $f^*$ is the pullback map of bundles induced by the forgetful map $f$ between the approximants $K_t$. Now let $\tilde{B}_t$ be the set of $t$ isomorphism classes of flat PE principal bundles with radius $t$ and let $\tilde{S}_t$ be the isomorphism classes of flat principal bundles over $K_t$. Inclusion maps $i$ induce the maps $\tilde{i}:\tilde{B}_s\rightarrow \tilde{B}_t$ for any $s\leq t$, by taking a representative from an equivalence class in $\tilde{B}_s$, mapping it by $i$ to $B_t$ and then taking the equivalence class of that in $\tilde{B}_t$. Similar considerations with the other maps in the above diagram will give us the following induced diagram:

\begin{diagram}
\tilde{B}_{t_1}     & \rInto^{\tilde{i}} & \tilde{B}_{t_2}      & \rInto^{\tilde{i}}    & \tilde{B}_{t_3}     & \rInto^{\tilde{i}}   & \cdots\\
\uTo_{\tilde{\pi}^*}  			&            & \uTo_{\tilde{\pi}^*}   	&           & \uTo_{\tilde{\pi}^*}	&    & \cdots\\
\tilde{S}_{t_1}     & \rTo_{\tilde{f^*}} & \tilde{S}_{t_2}& \rTo_{\tilde{f^*}} & \tilde{S}_{t_3} & \rTo_{\tilde{f^*}} & \cdots
\end{diagram}

Recall that we defined $PREP(T)$ by taking all flat principal PE bundles over $T$ and by modding this out by the equivalence relation given by identifying two bundles with radii $s$ and $t$ if they are eventually $t'$ isomorphic for some $t'$ bigger than both $s$ and $t$. Using the $\tilde{B}$s and $\tilde{i}$s above, this equivalence relation immediately becomes the direct limit:$$PREP(T)=\varinjlim_{\tilde{i}} \tilde{B}_{t_i}.$$ Let $\tilde{S}^*_t$ be the pullback of $\tilde{S}_t$ to $T$ by the projection map $\pi_t$. From section 2.2 we know that $\tilde{B}_t = \tilde{S}^*_t$. Now we can establish the main result we are after:

\begin{thm}
Let $T$ be a tiling space and $K_t$, $t\in\mathbb{R}$ be its G\"{a}hler approximants. Let $t_1\leq t_2\leq...$ be an increasing sequence and $f:K_{t_i}\rightarrow K_{t_j}$ be the forgetful map where $t_i$ is bigger than $t_j$. Then,
  \begin{equation}
  PREP(T) = \varinjlim_{f^*} Hom(\pi_1(K_{t_n}), G)/G.
  \end{equation}
\end{thm}

\begin{proof}
\begin{eqnarray}
PREP(T) & = & \varinjlim_{\tilde{i}} \tilde{B}_{t_n}\\
        & = & \varinjlim_{\tilde{i}}  \tilde{S}^*_t \\
        & = & \varinjlim_{\tilde{f}^*} \tilde{S}_t \\
        & = & \varinjlim_{f^*} Hom(\pi_1(K_{t_n}), G)/G.
\end{eqnarray}
\end{proof}

The corresponding result for the $PREP_b(T)$ is as follows:
\begin{thm}
Let $T$ be a tiling space and $K_t$, $t\in\mathbb{R}$ be its G\"{a}hler approximants. Let $t_1\leq t_2\leq...$ be an increasing sequence and $f:K_{t_i}\rightarrow K_{t_j}$ be the forgetful map where $t_i$ is bigger than $t_j$. Then,
  \begin{equation}
  PREP_b(T) = \varinjlim_{f^*} Hom(\pi_1(K_{t_n}), G).
  \end{equation}
\end{thm}

\section{Invariance of the Pattern Equivariant Representation Variety}\label{invariance}

In this section we will prove the invariance of the PE representation variety of two tiling spaces under homeomorphisms of tiling spaces. The key step will be the approximation theorem below which is proven in Rand's PhD thesis \cite{Rand}. First we need a definition:

\begin{defn}
Let $f:\Omega_T\rightarrow \Omega_{T'}$ be a map between two tiling spaces. $f$ is called a \textit{local map} if there exist an $r>0$ such that for every $T_1$, $T_2$ in $\Omega_T$, we have $T_1|_r=T_2|_r$ implies $f(T_1)|_s=f(T_2)|_s$ for some $s$.
\end{defn}

\begin{thm}
Let $X$ and $Y$ be tiling spaces. If f is a continuous map from $X\rightarrow Y$, then for every $\epsilon$ sufficiently small, there exists an $f_\epsilon$ with the following properties:
  \begin{enumerate}
    \item $f_\epsilon$ is a smooth, local map from $X\rightarrow Y$,
    \item $f(T)$ and $f_\epsilon(T)$ are in the same translational leaf, and $|f(T)-f_\epsilon(T)|<\epsilon$, where $|f(T)-f_\epsilon(T)|$ denotes the           translational gap within a leaf,
    \item If $f$ commutes with translations, then $f_\epsilon$ will also commute with translations.
  \end{enumerate}
\end{thm}

\begin{cor}
The maps $f$ and $f_\epsilon$ above are homotopic to each other.
\end{cor}

Now let $f:\Omega_T\rightarrow \Omega_{T'}$ be a continuous map. For a PE bundle $P$ over $T'$, we want to define its pullback but the problem is that $f$ is defined on the tiling space not the tiling itself. The standard trick is to restrict $f$ to the orbit $\mathbb{R}^2_T$ of $T$. Actually we want to restrict the approximation $f_\epsilon$ to the orbit of $T$ to get a map $f_\epsilon: \mathbb{R}^2_T\rightarrow \Omega_{T'}$. We will be using the same letter for a map and its restriction from now on. Since $f_\epsilon$ is continuous and $\Omega_{T'}$ is minimal, the image of $f_\epsilon$ is in the orbit of $T'$. Hence, we have a map $f_\epsilon:\mathbb{R}^2_T \rightarrow \mathbb{R}^2_{T'}$. This is the map we will use when we talk about the pullback of a bundle $P$ over $T'$ under $f:\Omega_T\rightarrow \Omega_{T'}$:	

\begin{equation}\label{eq:def}
f^* (P):=f_\epsilon^*(P).
\end{equation}

We have to make sure that a local map like $f_\epsilon$ pulls back PE bundles to PE bundles and have to show that the definition is independent of the choice of $f_\epsilon$.

\begin{lem}
Let $f:\Omega_T\rightarrow \Omega_{T'}$ be a local map. If $P'$ is a PE bundle over $T'$, then $P:=f^*(P')$ is a PE bundle over $T$.
\end{lem}

\begin{proof}
We need to check if PE transition functions pullback to PE transition functions and if PE open sets pullback to PE open sets. Let's start with the PE functions. Let $\alpha :\mathbb{R}^2_{T'}\rightarrow \mathbb{R}$ be a PE function with radius $s$. Since $f$ is local, there exist an $r>0$ such that  for every $T_1$, $T_2$ in the orbit of $T$ we have $T_1|_r=T_2|_r$ implies $f(T_1)|_s=f(T_2)|_s$. Now, for our tiling $T$, assume we have two patches with centers at $x$ and $y$ agreeing up to radius $r$. Define $T_1:=T-x$ and $T_2:=T-y$. Then, $f(T_1)|_s=f(T_2)|_s$. Hence, patches at $f(x)$ and $f(y)$ agree up to radius $s$, so that $\alpha(f(x))=\alpha(f(y))$. This means that $(f^*(\alpha))(x)=\alpha(f(x))=\alpha(f(y))=(f^*(\alpha))(y)$. So, $f^*(\alpha)$ is PE with radius $r$.

To see that PE open sets pullback to PE open sets, let $U'$ be a PE open set with radius $s$ in $\mathbb{R}^2_{T'}$. Let $x\in f^*(U')$ and let $y\in \mathbb{R}^2_T$ be such that the patches at the points $x$ and $y$ agree up to radius $r$ in $\mathbb{R}^2_T$ where we choose $r$ given by the locality of $f$. Then, by the similar arguments as above, we see that the patches around $f(x)$ and $f(y)$ agree up to radius $s$. Since $U'$ is PE with radius $s$, this means that $f(y)\in U'$. Hence, $y\in f^*(U')$ and hence $f^*(U')$ is PE with radius $r$.
\end{proof}

Now we have to show that the definition (\ref{eq:def}) is independent of the choice of $f_\epsilon$. If we instead defined it using an $f_\delta$, then since both are homotopic to $f$, they are homotopic to each other. The important point is we can approximate any homotopy between local maps by a local homotopy so that we stay in the category of the PE bundles when we pullback via this homotopy. The following lemma(see \cite{Rand}) shows that we can do this.

\begin{lem}
Let $F: \Omega_{T_0}\times [0, 1]\rightarrow \Omega_{T'_0}$ be a homotopy such that each $F_t$ are continuous maps on $\Omega_{T_0}$ and $F_0$ and $F_1$ are local maps. Then $F^*_0(PREP(\Omega_{T'_0}))=F^*_1(PREP(\Omega_{T'_0}))$.
\end{lem}

Using these results, we have the following.
\begin{thm}
If $f:\Omega_T\rightarrow \Omega_{T'}$ is a homeomorphism then $f^*$ is an isomorphism of pattern equivariant representation varieties of $T$ and $T'$.
\end{thm}

\section{Acknowledgments}
I would like to thank Lorenzo Sadun for his help during the whole process of writing this paper.

\section{Appendix: Proof of Lemma \ref{lem:ma232}}

\begin{figure}
\includegraphics[width=95mm]{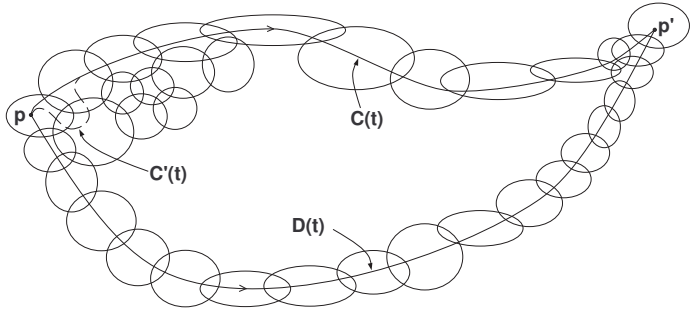} \caption{Independence of path}
\end{figure}

We will construct a map $\Psi: f^*(P)\rightarrow F$ such that when restricted to each fiber it is a homeomorphism. This implies that $f^*(P)$ is the trivial bundle $X'\times F$. Let $q\in f^*(P)$ be such that $\pi'(q)=p$ and $f(p)=\tilde{p}$. Given a chart $(U_{\tilde{p}}, \Phi_{\tilde{p}})$ around $\tilde{p}$, we can pull it back to a chart $(U_{p}, \Phi_{p})$ around $p$ and have the usual induced trivializing map $\phi_{p}:\pi'^{-1}(U_{p})\rightarrow F$ whose restriction to each fiber is a homeomorphism which we denote by the same symbol $\phi_{p}:\pi'^{-1}(p)\rightarrow F$. In the following, we will extend this map to all of $f^*(P)$. 

Let $q'\in f^*(P)$ be such that $\pi'(q')=p'$. If $p' \in U_{p}$ then we can use $\phi_{p}:\pi^{-1}(U_{p})\rightarrow F$ to map the fiber above $p'$ to $F$ and we are done. If $p' \notin U_{p}$ then cover the path $C:I\rightarrow X'$ from $p$ to $p'$ with finitely many open chart sets $\{U_{p_i}\}$, $i=0,1,2,\ldots,n$ where $U_{p_0}=U_{p}$ and $U_{p_n}=U_{p'}$. Then, multiplying the constant transition functions for each intersection $U_{p_i}\cap U_{p_j}$ we get a map $$\Psi: \pi^{-1}(U_{p'})\rightarrow F,$$ where $\Psi(r):=g_{pp_1}\cdot\;\cdots\;\cdot g_{p_{n-2}p_{n-1}}\cdot g_{p_{n-1}p'}\cdot \phi_{p'}(r)$. By the co-cycle condition it is easy to see that this map doesn't depend on the choice of ``\textsl{middle}'' open sets; that is, the open sets indexed from $i=1$ to $i=n-1$, but it depends on the choice of initial and final open sets.

Using the co-cycle condition and the simply connectedness of $X'$ we can show that $\Psi$ is independent of the path chosen to connect $p$ and $p'$ as follows: Let $D:I\rightarrow X'$ be another curve connecting these two points. Since these curves bound a compact region, we can fill it with finitely many open chart sets whose union contain the curves and the region they are bounding. Construct a local homotopy using three sets at a time and begin by perturbing the initial curve to a curve $C'$ as in the figure. By the simply connectedness of $X'$, we can have such a local homotopy and by the co-cycle condition the maps obtained from $C$ and $C'$ are the same. Repeat this procedure until $D$ is reached and patch these local homotopies to get a homotopy from $C$ to $D$. Hence our map is independent of the path and we have a global map $\Psi: f^*(P)\rightarrow F$ which is a homeomorphism when restricted to each fiber. Thus, $f^*(E)$ is the trivial bundle $X'\times F$.


\bibliographystyle{plain}

\end{document}